\documentclass[11pt,a4paper,reqno]{amsart}
\usepackage{amsfonts}
\usepackage{epsfig}
\usepackage{graphicx}
\usepackage{amsmath}
\usepackage{amssymb}
\usepackage{colortbl}

\newcommand{\Rr}{\mathbb{R}}
\newcommand{\Nn}{\mathbb{N}}

\newcommand{\Jj}{\mathbb{J}}

\newcommand{\id}{\operatorname{id}}

\newcommand{\te}[1]{\quad\text{#1}\quad}
\newcommand{\pde}[2]{\frac{\partial #1}{\partial #2}}
\newcommand{\norm}[2]{\left\| #1\right\|_{C^{#2}}}

\newcounter{main}

\newtheorem{theorem}{Theorem}[section]
\newtheorem{proposition}[theorem]{Proposition}
\newtheorem{lemma}[theorem]{Lemma}

\newtheorem{remark}{Remark}[section]

\newtheorem{maintheorem}{Theorem}

\newcommand{\blanksquare}{\,\,\,$\sqcup\!\!\!\!\sqcap$}

\newcounter{example}
{{\stepcounter{example}}{\flushleft {\bf Example \arabic{example}:}}}%
{\par}

\newcommand{\OO}{{\mathcal O}}

\title[Homoclinic tangencies in Hamiltonians]{Creation of homoclinic tangencies in Hamiltonians by the suspension of Poincar\'e sections}

\author[M. Bessa]{M\'{a}rio Bessa}
\address{Departamento de Matem\'atica, Universidade do Porto, 
Rua do Campo Alegre, 687, 4169-007 Porto, Portugal and ESTGOH-Instituto Polit\'ecnico de Coimbra, Rua General Santos Costa, 3400-124 Oliveira do Hospital, Portugal}
\email{bessa@fc.up.pt}

\author[J. Lopes Dias]{Jo\~{a}o Lopes Dias}
\address{Departamento de Matem\'atica and CEMAPRE, ISEG, 
Universidade T\'ecnica de Lisboa,
Rua do Quelhas 6, 1200-781 Lisboa, Portugal}
\email{jldias@iseg.utl.pt}

\begin{document}

\begin{abstract}
In this note we show that for any Hamiltonian defined on a symplectic 4-manifold $M$ and any point $p\in M$, there exists a $C^2$-close Hamiltonian  whose regular energy surface through $p$ is either Anosov or it contains a homoclinic tangency. Our result is based on a general construction of Hamiltonian suspensions for given symplectomorphisms on Poincar\'e sections already known to yield similar properties.
\end{abstract}

\date{\today}

\maketitle

\emph{keywords:} Hamiltonian vector field, Anosov flow, elliptic point, homoclinic tangency.\\

%%%%%%%%%%%%%%%%%%%%%%%%%%%%%%%%%%%%%%%%%%%%%%%%%%%
\section{Introduction and statement of the results}

A few years ago Palis conjectured that any dy\-na\-mi\-cal system can be approximated in a certain topology by a hyperbolic system without cycles, or by a system exhibiting either a homoclinic tangency or a heterodimensional cycle (cf.~\cite{Pa,P}).
Later, Pujals and Sambarino~\cite{PS} proved this conjecture for the $C^1$ to\-po\-lo\-gy in the context of diffeomorphisms on compact surfaces. 
Notice that there are no heterodimensional cycles for surface diffeomorphisms. 

A version for flows appeared in~\cite{AH} stating that on a 3-dimensional compact manifold, a vector field can be $C^1$-approximated by another satisfying only one of the following phenomena:
\begin{itemize}
 \item  uniform hyperbolicity with no cycles,
 \item  a homoclinic tangency,
 \item  a singular cycle.
\end{itemize}
It has been further conjectured~(\cite[Conjecture 4]{P}) that the last situation above can be replaced by a singular hyperbolic set (see ~\cite{MPP} for the definition). 

Related results can be obtained when restricting to conservative systems. In fact, any divergence-free vector field defined on a 3-dimensional closed manifold can be $C^1$-approximated in the same class by a vector field either Anosov or with a homoclinic tangency associated to a hyperbolic closed orbit~\cite{BR2}. This was recently generalized in~\cite{F} for a $d$-dimensional closed manifold, $d\geq 4$: any divergence-free vector field can be $C^1$-approximated by another one satisfying either one of the properties of the 3-dimensional case, or with a heterodimensional cycle.
In this note we address the problem of obtaining a version of~\cite{BR2} in the Hamiltonian context.

Let $(M,\omega)$ be a compact symplectic $C^\infty$ $2d$-manifold, $d\geq2$, with a smooth boundary $\partial M$. Let $C^s(M)$, $2\leq s\leq\infty$, stand for the set of $C^s$ real-valued functions on $M$ constant on each connected component of $\partial M$, which we call $C^s$-Hamiltonians. 
We endow $C^s(M)$ with the $C^r$-Whitney topology.
For each $H\in C^s(M)$ one has the Hamiltonian vector field $X_H$ and the Hamiltonian flow $\varphi_H^t$. 
Consider an \textit{energy} $e\in H(M)\subset \Rr$ and the associated $\varphi_H^t$-invariant \textit{energy level set} $H^{-1}(e)$. 
An \emph{energy surface} is a connected component of $H^{-1}(e)$.
We say that it is regular if it does not contain critical points.

A regular energy surface is \emph{Anosov} if it is uniformly hyperbolic (cf. \cite{BFR}). 
It is {\it far from Anosov} if it is not in the closure of Anosov regular energy surfaces.
Moreover, Anosov regular energy surfaces do not contain singularities or elliptic closed orbits.

Let us state the main result in this note.

\begin{maintheorem}\label{New}
Let $d=2$, $H\in C^2(M)$ and $p\in M$.
There exists a Hamiltonian $C^2$-close to $H$ whose regular energy surface through $p$ is either Anosov or else it contains a homoclinic tangency associated to some hyperbolic closed orbit. 
\end{maintheorem}

Recall that the existence of homoclinic tangencies is a sufficient condition to have elliptic points (see~\cite{N,Du}). 
We see that it is also a necessary condition for, at least, a sufficient $C^1$-close vector field.

\begin{maintheorem}\label{New2}
Let $d=2$, $H\in C^2(M)$ and $p\in M$ lies in an elliptic closed orbit of $H$. Then, there exists a Hamiltonian $C^2$-close to $H$ whose regular energy surface through $p$ has a homoclinic tangency associated to some hyperbolic closed orbit.
\end{maintheorem}

In the proof of Theorem~\ref{New2} (section~\ref{end}) we apply a mechanism introduced in~\cite{GT} to create homoclinic intersections by perturbations of area-preserving maps with elliptic points (see section~\ref{sec GT}). 
We use that in our context by finding a Hamiltonian flow (through Theorem~\ref{thm suspension} below) that yields a Poincar\'e map with the same properties -- see section~\ref{sec Ham realization}. 
Theorem \ref{New} is then a direct consequence of Theorem~\ref{New2} and of the Newhouse dichotomy (Theorem \ref{BD}).

The last result in this note is a Hamiltonian suspension theorem, especially useful for the conversion of perturbative results between symplectomorphisms and Hamiltonian flows in any dimension $2d$. 
Indeed, if we perturb the Poincar\'e map of a periodic orbit (cf. section~\ref{sec Poincare maps}), there is a nearby Hamiltonian realizing the new map.

\begin{maintheorem}[Hamiltonian suspension]\label{thm suspension}
Let $d\geq 2$ and $H\in C^\infty(M)$ with Poincar\'e map $f$ at a periodic point $p$.
Then, for any $\epsilon>0$ there is $\delta>0$ such that for any symplectomorphism $\widetilde f$ being $\delta$-$C^3$-close to $f$, there is a Hamiltonian $\widetilde H$ $\epsilon$-$C^2$-close with Poincar\'e map $\widetilde f$.
\end{maintheorem}

The proof of the above theorem is contained in section~\ref{sec Ham realization}.
It is based on the construction using generating functions of an isotopy between $f$ and $\widetilde f$, that extends to a Hamiltonian flow with the required properties.
This type of suspension of Poincar\'e maps is already mentioned in~\cite{Douady} when the manifold is the annulus (see also~\cite{Channell}), but without an explicit construction.

We remark that the result in ~\cite{GT} holds also for real-analytic Hamiltonians. 
However, the problem of suspending a real-analytic Poincar\'e map into a Hamiltonian flow is of a very different sort because of the lack of real-analytic bump functions, and remains an open problem.
So, in this case, it is required to find versions of the pertubation results directly for flows.

%%%%%%%%%%%%%%%%%%%%%%%%%%%%%%%%%%%%%%%%%%%%%%%
\section{Preliminaries}

In this section we assume $(M,\omega)$ to be a symplectic $2d$-manifold, with $d\geq2$. 

%%%%%%%%%%%%%%%%%
\subsection{Poincar\'e maps}\label{sec Poincare maps}

Consider $H\in C^2(M)$ and a closed orbit $\mathcal{O}$ with least period $T>0$ for $\varphi_H^t$.
At a point $p\in\mathcal{O}$ consider a transversal $\Sigma\subset M$ to the flow, i.e. a local $(2d-1)$-submanifold for which $X_H$ is nowhere tangencial.
By choosing $e=H(p)$, define the dimension $2d-2$ symplectic submanifold
$$
\Sigma_e=\Sigma\cap H^{-1}(\{e\}).
$$
Thus, for any $x\in\Sigma_e$,
$$
T_xH^{-1}(\{e\})=T_x\Sigma_e\oplus\Rr X_H(x),
$$
where $\Rr X_H(x)$ stands for the flow direction.

Let $U\subset M$ be some open neighbourhood of $p$ and $V=U\cap\Sigma_e$.
The {\it Poincar\'e (section) map} $f\colon V\to\Sigma_e$ is the return map of $\varphi_H^t$ to $\Sigma_e$. It is given by
$$
f(x)=\varphi_H^{\tau(x)}(x),
\qquad
x\in V,
$$
where $\tau$ is the return time to $\Sigma_e$ defined implicitely by the relation $\varphi_H^{\tau(x)}(x)\in\Sigma_e$ and satisfying $\tau(p)=T$.
In addition, $p$ is a fixed point of $f$. 
Notice that one needs to assume that $U$ is a small enough neighbourhood of $p$. Thus, $f$ is a $C^{1}$-symplectomorphism between $V$ and its image.
Moreover, any two Poincar\'e section maps of the same closed orbit are conjugate by a symplectomorphism.

%%%%%%%%%%%%%%%%%%%%%%%%%%%%%%%%%%%
\subsection{Homoclinic tangencies}

Take $H\in C^2(M)$, a non-constant hyperbolic closed orbit $\mathcal{O}$ and a transversal section at a point $p\in\mathcal{O}$.
Let $W^s_p$ be the stable manifold at $p$ of the Poincar\'e map, and $W^u_p$ the unstable manifold. 
We say that $\mathcal{O}$ has a \emph{homoclinic tangency} at $q\not=p$ if the invariant manifolds $W^s_p$ and $W^u_p$ have a non transversal intersection, i.e.:
\begin{itemize}
 \item $T_q W^s_p\cap T_q W^u_p$ contains a nonzero vector,
 \item $T_q W^s_p\oplus T_q W^u_p\oplus \Rr X(q)\not= T_q H^{-1}(p)$.
\end{itemize}

%%%%%%%%%%%%%%%%%%%%%%%%%%%%%%%%%%%
\subsection{Hamiltonian flowtube coordinates}\label{LFC}

Denote the coordinates in $\Rr^{2d}$ as $(x_1,\dots,x_d,y_1,\dots,y_d)$.
The canonical symplectic form is given by
\begin{equation*}\label{canonical symplectic form}
\omega_0=\sum_{i=1}^d dx_i\land dy_{i}.
\end{equation*}
The Hamiltonian vector field of any smooth Hamiltonian $H$ on $(\Rr^{2d},\omega_0)$ is then 
$$
X_H=\Jj\nabla H,
$$
where 
$\Jj=\left(\begin{smallmatrix}0& I \\
-I&0\end{smallmatrix}\right)$ 
and $I$ is the $d\times d$ identity matrix.

Consider $H_0\colon\Rr^{2d}\to\Rr$ given by $H_0= y_{d}$, so that 
$$
X_{H_0}=\frac{\partial}{\partial x_d}.
$$
Hence, the flow is $\varphi_{H_0}^t=\id+(0,\dots,t,0,\dots,0)$.

The following results provide us with the above coordinates, useful to perform local perturbations of a Hamiltonian defined on any symplectic manifold $(M,\omega)$.

\begin{theorem}[Hamiltonian flowbox, cf.~e.g.~\cite{BD2}]\label{robinson}
Let $H\in C^s(M)$, $s\geq2$ or $s=\infty$, and $p\in M$. If $dH(p)\not=0$, there exists a neighborhood $U\subset M$ of $p$ and a local $C^{s-1}$-symplectomorphism $g\colon (U,\omega)\to(\Rr^{2d},\omega_0)$ such that $H=H_0\circ g$ on $U$.
\end{theorem}

By considering neighbourhoods as above taken along a piece of a trajectory, we can find a small tubular neighborhood where the flow is again straightened. This is the content of the next result.

\begin{theorem}[Hamiltonian flowtube]\label{flowtube}
Let $H\in C^s(M)$, $s\geq2$ or $s=\infty$, and a non-closed compact self-avoiding arc of trajectory $\Gamma\subset M$. 
There exists a neighborhood $W\subset M$ of $\Gamma$ and a local $C^{s-1}$-symplectomorphism $\phi\colon (W,\omega)\to(\Rr^{2d},\omega_0)$ such that $H=H_0\circ \phi$ on $W$.
\end{theorem}

%%%%%%%%%%%%%%%%%%%%%%%%%%%%
\subsection{Density of elliptic closed orbits}

The next result is the Hamiltonian version of the Newhouse dichotomy~\cite{N} for 4-dimensional Hamiltonians. 
As previously mentioned, it will be used in the proof of Theorem~\ref{New} (see section~\ref{end}).

\begin{theorem}[\cite{BD}]\label{BD}
Let $d=2$.
Given an open set $U\subset M$ intersecting a far from Anosov regular energy surface of $H\in C^2(M)$, there is a $C^2$-nearby Hamiltonian having an elliptic closed orbit through $U$. Moreover, this implies that, for far from Anosov regular energy surfaces of a $ C^2$-generic Hamiltonian, the elliptic closed orbits are dense.
\end{theorem}

%%%%%%%%%%%%%%%%%%%%%%%%%%%%
\subsection{Creation of homoclinic tangencies}\label{sec GT}

The next result is central to the proof of Theorem~\ref{New}.
It deals with symplectomorphisms on a symplectic 2-manifold, i.e. area-preserving maps.

\begin{theorem}[Gelfreich and Turaev~\cite{GT}]\label{GT}
Let $r\in\Nn\cup\{\infty,\omega\}$.
Any $C^r$-area-preserving map with an elliptic point can be $C^r$-approximated by another area-preserving map with a homoclinic tangency.
\end{theorem}

%%%%%%%%%%%%%%%%%%%%%%%%%%%%%%%%%%%%%%
\section{Hamiltonian realization of a perturbed Poincar\'e map}\label{sec Ham realization}

Consider a Hamiltonian flow with a closed orbit and an associated Poincar\'e section map in an energy surface.
Our goal in this section is to find a nearby Hamiltonian exhibiting a perturbed Poincar\'e map (Theorem~\ref{thm suspension}).
In order to prove Theorem~\ref{New2}, we will only make use of the case $d=2$.
Nevertheless, we study here the general situation for future use.

%%%%%%%%%
\subsection{Suspension of Poincar\'e maps}

Let $H\in C^\infty(M)$.
Consider a closed orbit $\OO$ with least period $T>0$, $p\in\OO$ and $e=H(p)$.
The Poincar\'e map is given by $f\colon V\to\Sigma_e$ as in section~\ref{sec Poincare maps}, having a fixed point at $p$.

The return time $\tau\colon V\to\Rr^+$ is close to $T$. 
So, choose $T_0,T_1>0$ such that $T_0+T_1\leq\frac12\min\{\tau(x)\colon x\in V\}$.
%In this way, we write
%$$
%\Sigma_{0}=\varphi_H^{T_0}(V)
%\quad\text{and}\quad
%\Sigma_{1}=\varphi_H^{-T_1}\circ f(V).
%$$
%These are transversals to the flow around $\OO \cap H^{-1}(e)$ at two distinct points.
%
Take the arc of trajectory
$$
\Gamma=\{\varphi_H^t(p)\colon T_0\leq t\leq T-T_1\}\subset\OO.
$$
By Theorem~\ref{flowtube}, in a tubular neighbourhood $W\subset M$ of $\Gamma$ we have $H=H_0\circ\phi$.
One can always compose $\phi$ with some symplectomorphism $\psi$ so that $S_0,S_1\subset \psi\circ\phi(W)$, where 
$$
S_0=\{(x_1,\dots,x_d,y_1,\dots,y_d)\in\Rr^{2d}\colon x_d=y_d=0\}
$$
and $S_1=\varphi_{H_0}^1(S_0)$.
We assume that $\phi$ is in fact $\psi\circ\phi$ in order to simplify notations.
Furthermore,
$$
\varphi_{H_0}^1|S_0=\phi\circ \varphi_H^{-T_1}\circ f\circ\varphi_H^{-T_0}\circ \phi^{-1},
$$
which is simply given by $\varphi_{H_0}^1(x,0,y,0)=(x,1,y,0)$ with 
$$
(x,y)=(x_1,\dots,x_{d-1},y_1,\dots,y_{d-1})\in\Rr^{2d-2}.
$$
This means that $\Pi\circ \varphi_{H_0}^1|S_0=\id$ by using the projection $\Pi\colon\Rr^{2d}\to\Rr^{2d-2}$, $(x,x_d,y,y_d)\mapsto(x,y)$.

Given a $C^\infty$-symplectomorphism $\widetilde f$ on $V$ that is $C^1$-close to $f$, we want to find a Hamiltonian $\widetilde H$ having $\widetilde f$ as Poincar\'e map.
The perturbation is constructed inside $W$, hence being enough to find $\widetilde H_0=\widetilde H\circ \phi^{-1}$ such that 
$$
\varphi_{\widetilde H_0}^1|S_0=\phi\circ \varphi_H^{-T_1}\circ \widetilde f\circ\varphi_H^{-T_0}\circ \phi^{-1}.
$$
Then, $g=\Pi\circ \varphi_{\widetilde H_0}^1|S_0$ is a $C^\infty$-symplectomorphism on $\Rr^{2d-2}$.
From the above considerations we know that for any $r\geq0$,
$$
\norm{g-\id}r\leq c\|\widetilde f-f\|_{C^r}
$$
for some $c_r>0$ depending on $H$.

Let $\rho>0$ and the euclidean open ball
$$
B_\rho=\{(x,y)\in\Rr^{2d-2}\colon \|(x,y)\|<\rho\}.
$$
The radius $\rho$ is chosen small enough so that $B_\rho\times\{0\leq x_d\leq1,|y_d|<\rho\}\subset \phi(W)$.

\begin{proposition}\label{thm pert}
There is $\delta,c>0$ such that for any $C^\infty$-symplectomorphism $g$ compactly supported in $B_\rho$, $\delta$-$C^{1}$-close to the identity, we can find $\widetilde H_0\in C^\infty(\Rr^{2d})$ compactly supported in $B_\rho$ verifying
$$
\Pi \circ \varphi_{\widetilde H_0}^1|{S_0}=g
$$
and
\begin{equation}\label{bound on Ham0}
\norm{\widetilde H_0-H_0}2 \leq c (1+\rho+\rho^{-1}+\rho \|g-\id\|_{C^3}^2) \,\norm{g-\id}1.
\end{equation}
Moreover, if $g$ fixes the origin, then $\varphi_{\widetilde H_0}^1(0)=(0,1,0,0)$.
\end{proposition}

We now use the above proposition (to be proved in section~\ref{sect proof prop} below) to complete the proof of Theorem~\ref{thm suspension}.
Consider
$$
\widetilde H=
\begin{cases}
H,& \text{on }M\setminus W\\
H+(\widetilde H_0-H_0)\circ\phi,& \text{otherwise}.
\end{cases}
$$
Therefore, combining the estimates above and assuming that $\widetilde f$ is $C^3$-close to $f$, one gets
$$
\|\widetilde H-H\|_{C^2}\leq c \|\widetilde f-f\|_{C^{1}}
$$
for some $c>0$.

%%%%%%%%%%%%%%%%%%%
\subsection{Proof of Proposition~\ref{thm pert}}\label{sect proof prop}

Since the group of smooth symplectomorphisms isotopic to the identity is path-connected, we can always find an isotopy $g_\alpha$, $\alpha\in[0,1]$, of symplectomorphisms from the identity to $g$. 
The corresponding non-autonomous vector field $X_\alpha=\dot g_\alpha\circ g_\alpha^{-1}$ is symplectic (for each $\alpha$), and in fact Hamiltonian since we are in a simply connected space.
The proof of Proposition~\ref{thm pert} relies on this well-known fact, but it also requires a control on the size of the derivatives of $(x,y,\alpha)\mapsto g_\alpha(x,y)$.
For this reason we need to construct $g_\alpha$ through a simple isotopy of generating functions, whose norms are easily estimated.
Later, by adding a flow direction coordinate ($\alpha=x_d$) and its symplectic conjugate (the ``energy'' $y_d$), we will extend our Hamiltonian to $\mathbb{R}^{2d}$.

For functions $F\colon D\to\Rr^m$, $D\subset \Rr^{2d}$, consider the $C^s$-norm, with $s\in\Nn_0=\Nn\cup\{0\}$,
$$
\|F\|_{C^s}=
\max_{i=1,\dots,m}
\max_{|\sigma|\leq s} \sup_{D}
\left|
\frac{\partial^{|\sigma|}F_i}
{\partial^{\sigma_1}x_1\dots \partial^{\sigma_{2d}}y_{d}}
\right|
$$
where $\sigma=(\sigma_1,\dots,\sigma_{2d})\in\Nn_0^{2d}$ and $|\sigma|=\sum_i\sigma_i$.
Moreover, $\langle\cdot,\cdot\rangle$ denotes the usual euclidean scalar product and we introduce the projections $\pi_1(x,y)=x$ and $\pi_2(x,y)=y$.

Let $V\in C^\infty(\Rr^{2d-2})$ such that 
$$
W(x',y)=\langle x' , y \rangle+V(x',y)
$$
is a generating function of $g$.
More specifically, writing $(x',y')=g(x,y)$, since $\det D_1x'\not=0$,
$$
x=\pde Wy(x',y)
\te{and}
y'=\pde W{x'}(x',y).
$$
Therefore,
$$
g(x,y)=(x,y)-\Jj\nabla V\circ G(x,y).
$$
where $G(x,y)=(\pi_1 g(x,y),y)$ and $\|\nabla V\|_{C^0}=\|g-\id\|_{C^0}$.
We assume that $g$ is sufficiently $C^1$-close to the identity, thus $G$ is a diffeomorphism.

\begin{lemma}\label{lemma bdd on nabla V}
For $r\geq 1$, there is $c_r>0$ such that
$$
\|\nabla V\|_{C^r}\leq c_r \max\{1,\|G^{-1}\|_{C^r}^r\} \, \|g-\id\|_{C^r}.
$$
\end{lemma}

\begin{proof}
Write $\phi=g-\id$ and $\beta=G^{-1}$ so that $\phi\circ \beta=-\Jj\nabla V$.
Recall the Fa\`a di Bruno formula for the higher derivative chain rule:
\begin{equation}\label{faa di bruno formula}
D^r(\phi\circ\beta)= \sum \frac{r!}{k_1!\dots k_r!1!^{k_1}\dots r!^{k_r}}
D^{|k|}\phi(\beta)\,(\underbrace{D\beta,\dots,D\beta}_{k_1},\dots,\underbrace{D^r\beta,\dots,D^r\beta}_{k_r})
\end{equation}
where the sum is over every $k=(k_1,\dots,k_r)\in\Nn_0^r$ such that 
$$
\langle k,(1,2,\dots,r)\rangle=r.
$$
Therefore, there is a constant $c_r>0$ depending on $r$, satisfying
$$
\|\nabla V\|_{C^r} \leq
c_r \, \max\{1,\|\beta\|_{C^r}^r\} \, \|\phi\|_{C^r} ,
$$
where we have used that $\|\beta\|_{C^{k_i}}^{k_i}\leq \|\beta\|_{C^r}^{k_i} \leq \max\{1,\|\beta\|_{C^r}^r\}$.
\end{proof}

Let $\ell\in C^\infty(\Rr)$ be a bump function verifying
$$
\ell(\alpha)=
\begin{cases}
1,& \alpha\geq \xi \\
0,& \alpha\leq 0
\end{cases}
$$ 
for some choice of $0<\xi<1$ such that $\ell'>0$ in $(0,\xi)$.
We can now construct the following smooth 1-family of generating functions:
$$
W_\alpha(x',y)=\langle x' , y \rangle+\ell(\alpha)\,V(x',y).
$$
For each $\alpha\in\Rr$ we obtain a $C^\infty$-symplectomorphism $g_\alpha$ generated by $W_\alpha$. 
Clearly, $g_0=\id$ and $g_1=g$.
Hence, $g_\alpha$ is a $C^\infty$-isotopy between $\id$ and $g$ implicitly given by
$$
g_\alpha= \id - \ell(\alpha) \, \Jj\nabla V\circ G_\alpha,
$$
where $G_\alpha=(\pi_1 g_\alpha,\pi_2)$ and $\|g_\alpha-\id\|_{C^0}\leq \|\nabla V\|_{C^0} =\|g-\id\|_{C^0}$.

\begin{lemma}\label{lemma bdd on nabla V 2}
For $r\geq1$, there is $c_r>0$ such that for any $\alpha\in\Rr$, if $\|g-\id\|_{C^1}$ is sufficiently small, then
$$
\|g_\alpha-\id\|_{C^r}\leq \frac{c_r}{1-\|\nabla V\|_{C^1}}  \, \|g-\id\|_{C^{r-1}}^r\|\nabla V\|_{C^r}.
$$
\end{lemma}

\begin{proof}
Write $v_\alpha=-\ell(\alpha)\,\Jj\nabla V$ so that $\|v_\alpha\|_{C^r}\leq\|\nabla V\|_{C^r}$.
Using again the Fa\`a di Bruno formula,
\begin{equation*}
\begin{split}
D^r(g_\alpha-\id) & =\sum_{k_r=0} c_{k,r}
D^{|k|}v_\alpha(G_\alpha)\,
(\underbrace{DG_\alpha,\dots,DG_\alpha}_{k_1},\dots,\underbrace{D^{r-1}G_\alpha,\dots,D^{r-1}G_\alpha}_{k_{r-1}}) \\
& +Dv_\alpha(G_\alpha)\,D^rG_\alpha,
\end{split}
\end{equation*}
where $c_{k,r}$ are the coefficients as in~\eqref{faa di bruno formula} and we have split the sum in the terms corresponding to the vectors $k=(k_1,\dots,k_{r-1},0)$ and $k=(0,\dots,0,1)$.
Taking the norms, with $c_r>0$ depending on $r$,
$$
\|g_\alpha-\id\|_{C^r}\leq c_r \|v_\alpha\|_{C^{r}}\|g_\alpha-\id\|_{C^{r-1}}^{r}+\|v_\alpha\|_{C^1}\,\|g_\alpha-\id\|_{C^r}.
$$
Therefore,
$$
\|g_\alpha-\id\|_{C^r}\leq \frac{c_r}{1-\|v_\alpha\|_{C^1}}
\|g_\alpha-\id\|_{C^{r-1}}^{r}\|v_\alpha\|_{C^{r}}.
$$
The claim follows from applying Lemma~\ref{lemma bdd on nabla V}.
\end{proof}

Consider now the $C^\infty$-vector field $\dot g_\alpha=\frac{d}{d\alpha}g_\alpha$ on $\Rr^{2d-2}$ that generates the isotopy $g_\alpha$.
The non-autonomous vector field 
$$
X_\alpha=\dot g_\alpha\circ g_\alpha^{-1}
$$ 
is symplectic, i.e. $\iota_{X_\alpha}\omega_0$ is a closed 1-form. 
By the Poincar\'e lemma, since our space is simply-connected, it is also exact.
Therefore, for each $\alpha$ there exists a $C^\infty$-function $K_\alpha\colon\Rr^{2d-2}\to\Rr$ with compact support such that $\iota_{X_\alpha}\omega_0=d K_\alpha$,
i.e. $\nabla K_\alpha=-\Jj X_\alpha$ and using the notation of a Hamiltonian vector field 
$$
X_{K_\alpha}=X_\alpha.
$$
Up to a constant (chosen so that $K_\alpha$ has compact support), it is given by
\begin{equation}\label{def K}
 K_{\alpha}(x,y)=\int_{[0,(x,y)]}\iota_{X_\alpha}\omega_0=\int_0^1 \left\langle X_{K_\alpha}(s(x,y)),(y,-x)\right\rangle \, ds,
\end{equation}
where the integration is along the straight path $[0,(x,y)]$ that connects $(x,y)$ to the origin.
Notice that the vector field that determines $g$ as the time-1 map is non-autonomous, not preserving the ``energy'' $K$.
Also, $K_\alpha=0$ for any $\alpha\not\in(0,1)$.

We can extend the dimension of the space to $\Rr^{2d}$ by considering the variables $x_d=\alpha$ (seen as the time direction) and $y_d$ (the ``energy'' $K$).

Let $\widetilde\ell\in C^\infty(\Rr)$ be another bump function satisfying
$$
\widetilde \ell(y_d)=
\begin{cases}
1,& |y_d|\leq \nu\rho\\
0,& |y_d|\geq \rho
\end{cases}
$$ 
for any choice of $0<\nu<1$, such that $\|\widetilde\ell\|_{C^0}\leq1$, 
$$
\|\widetilde\ell'\|_{C^0}\leq\frac2{(1-\nu)\rho}
\te{and}
\|\widetilde\ell''\|_{C^0}\leq\frac4{(1-\nu)\rho^2}.
$$
We define the (autonomous) $C^\infty$-Hamiltonian $\widetilde H_0\colon \Rr^{2d}\to\Rr$ as
$$
\widetilde H_0(x,x_d,y,y_d)=H_0(y_d)+K_{x_d}(x,y) \,\widetilde\ell(y_d)
$$
with $H_0(y_d)=y_d$. 
Hence,
\begin{equation}\label{def grad H}
\nabla(\widetilde H_0-H_0)=\left(\widetilde\ell\,\pde Kx,\widetilde\ell\,\pde K{x_d},\widetilde\ell\,\pde Ky,\widetilde\ell' \,K
\right).
\end{equation}
Notice that outside $\{x_d\in(0,1), |y_d|<\rho\}\subset\Rr^{2d}$ we have $\widetilde H_0=H_0$. By contrast, the Hamiltonian vector field for $x_d\in[0,1]$ and $|y_d|\leq\nu\rho$ is
$$
X_{\widetilde H_0}=\left(\pi_1 X_K,1, \pi_2 X_K,-\pde{K}{x_d}\right).
$$

\begin{lemma}\label{lemma est vf}
There is $\delta>0$ and $c>0$ such that, if $\|g-\id\|_{C^1}\leq\delta$, then~\eqref{bound on Ham0} holds.
\end{lemma}

\begin{proof}
 We write a dot to represent the derivative with respect to $x_d$ and $D$ for the derivative with respect to $(x,y)$.
Recall that $X_K(x,x_d,y,y_d)=\dot g_{x_d}\circ g_{x_d}^{-1}(x,y)$.
We will use Lemmas~\ref{lemma bdd on nabla V} and~\ref{lemma bdd on nabla V 2} without explicit mention.

From~\eqref{def grad H} we have
$$
\norm{\widetilde H_0-H_0}1\leq 
\max\left\{ \norm K0,\norm{X_K}0,\|\dot K\|_{C^0}, \|\widetilde\ell'\|_{C^0}\norm K0\right\}.
$$
Now, the second order derivatives of $\widetilde H_0$ are
\begin{equation*}
\begin{split}
\frac{\partial^2\widetilde H_0}{\partial z_i\partial z_j}&=\widetilde\ell\, \frac{\partial^2K}{\partial z_i\partial z_j} \\
\frac{\partial^2\widetilde H_0}{\partial z_i\partial x_d}&=\widetilde\ell \,\frac{\partial\dot K}{\partial z_i}\\
\frac{\partial^2\widetilde H_0}{\partial^2 x_d}&=\widetilde\ell \,\ddot K \\
\frac{\partial^2\widetilde H_0}{\partial z_i\partial y_d}&=\widetilde\ell' \,\frac{\partial K}{\partial z_i}\\
\frac{\partial^2\widetilde H_0}{\partial x_d \partial y_d}&=\widetilde\ell' \,\dot K \\
\frac{\partial^2\widetilde H_0}{\partial^2 y_d}&=\widetilde\ell'' \, K
\end{split}
\end{equation*}
where $z=(x,y)$ and $i,j=1,\dots,2d-2$.
So,
\begin{equation*}
\begin{split}
\norm{\widetilde H_0-H_0}2\leq 
\max\left\{ \right.
&
\norm{X_K}1,\|\widetilde\ell'\|_{C^0}\norm{X_K}0,\|\ddot K\|_{C^0}, \\
& 
\max\{1,\|\widetilde\ell'\|_{C^0}\}\|\dot K\|_{C^0}, \\
& \left.
\max\{1,\|\widetilde\ell'\|_{C^0},\|\widetilde\ell''\|_{C^0}\}\norm K0
\right\}.
\end{split}
\end{equation*}

By writing $v=-\Jj\nabla V$, we have that 
$$
\|\dot g\|_{C^0}\leq 
\norm\ell1\,\|v\|_{C^0}+\|v\|_{C^1} \|\dot g\|_{C^0}.
$$ 
Therefore,
$$
\|\dot g\|_{C^0}\leq
\frac{\norm\ell1 \|g-\id\|_{C^0}}{1-\norm v1}
\leq
c\norm{g-\id}0
$$
for some $c>0$.
Similarly,
\begin{equation*}
\begin{split}
\|\ddot g\|_{C^0} 
&\leq 
\frac{\norm\ell2 \norm v0+2\norm\ell1 \norm v1 \|\dot g\|_{C^0} + \norm v2 \|\dot g\|_{C^0}^2  }
{1-\norm v1} \\
&\leq 
c\norm{g-\id}0
\end{split}
\end{equation*}
for some $c>0$.
Moreover,
$$
\|D\dot g\|_{C^0} \leq 
\norm\ell1\,\norm v1\norm{g}1 + \|v\|_{C^2} \norm{g}1 \|\dot g\|_{C^0} + \norm{v}1\|D\dot g\|_{C^0},
$$
thus
\begin{equation*}
\begin{split}
\|D\dot g\|_{C^0} 
&\leq 
\frac{\norm\ell1\,\norm v1\norm{g}1 + \|v\|_{C^2}  \|\dot g\|_{C^0} \norm{g}1}
{1-\norm v1} \\
&\leq 
c\norm{g-\id}1
\end{split}
\end{equation*}
for some $c>0$.

From $\dot X_K=\ddot g\circ g^{-1}+D\dot g\circ g^{-1}\, \dot g^{-1}$ and $DX_K=D\dot g\circ g^{-1}\,Dg^{-1}$,  
$$
\norm{X_K}1\leq c\norm{g-\id}1.
$$
From~\eqref{def K}, $\norm K0\leq\rho\norm{X_K}0$, $\|\dot K\|_{C^0}\leq\rho\|X_K\|_{C^1}$ and also $\|\ddot K\|_{C^0}\leq\rho\|\ddot X_K\|_{C^0}$.
Thus, it remains to bound $\|\ddot X_K\|_{C^0}$.

As before, we obtain the following bounds:
\begin{equation*}
\begin{split}
\|\dddot g\|_{C^0}\leq
\frac1{1-\norm v1}
\left( \right.&
\norm\ell3 \norm v0 +3\norm\ell2\norm v1\|\dot g\|_{C^0} \\
& 
+3\norm\ell1\norm v2\|\dot g\|_{C^0}^2 \\
& \left.
+3\norm\ell1\norm v1\|\ddot g\|_{C^0} +
\norm v3\|\dot g\|_{C^0}^3
\right) \\
\|D^2\dot g\|_{C^0} \leq
\frac1{1-\norm v1}
\left( \right.&
\norm\ell1 \norm v1 \|g\|_{C^1}^2+\norm\ell1\norm v1\|D^2g\|_{C^0} \\
& 
+\norm v3\|g\|_{C^1}^2\|\dot g\|_{C^0} 
+\norm v2\|D^2g\|_{C^0}\|\dot g\|_{C^0} \\
& \left. 
+2\norm v2\norm g1 \|\dot g\|_{C^1}
\right) \\
\|D\ddot g\|_{C^0}\leq
\frac1{1-\norm v1}
\left( \right.&
\norm\ell2 \norm v1 {\norm g1}+2\norm\ell1\norm v2\norm g1\|\dot g\|_{C^0} \\
& 
+2\norm\ell1\norm v1\|\dot g\|_{C^1} +\norm v3\|g\|_{C^1}\|\dot g\|_{C^0}^2 \\
& \left.
+2\norm v2 \|\dot g\|_{C^1}\|\dot g\|_{C^0}
+\norm v2\norm g1\|\ddot g\|_{C^0}
\right)
\end{split}
\end{equation*}

Finally, we use the fact that $\ddot X_K=\dddot g\circ g^{-1}+2D\ddot g\circ g^{-1}\,\dot g^{-1} +D^2\dot g\circ g^{-1}\,(\dot g^{-1},\dot g^{-1}) + D\dot g\circ g^{-1}\,\ddot g^{-1}$.
So,
$$
\|\ddot X_K\|_{C^0}\leq c\, \left(1+\|g-\id\|_{C^3}^2\right)\norm{g-\id}1
$$
for some constant $c>0$.
Evaluating all the above estimates together, one gets
$$
\norm{\widetilde H_0-H_0}2\leq c\,\left(1+\rho+\rho^{-1}+\rho\|g-\id\|_{C^3}^2
\right) \norm{g-\id}1
$$
for some universal constant $c>0$ that only depends on the norms of the bump functions.
\end{proof}

\begin{remark}
In the above lemma there is the need to bound the size of higher derivatives of $g$.
This loss of differentiability is caused by our specific construction of the isotopy $g_\alpha$. 
It should be possible to use a different isotopy that avoids this phenomenon.
Our choice was done for the sake of simplicity, since it does not restrict our main results.
\end{remark}

The Hamiltonian flow for $x_d\in[0,1]$ and $|y_d|\leq\nu\rho$ is given by
\begin{equation*}
\begin{split}
\varphi_{\widetilde H_0}^t(x,x_d,y,y_d)
=& \left(\pi_1g_{x_d+t}\circ g_{x_d}^{-1}(x,y), \right.\\
& x_d+t,\\
& \pi_2g_{x_d+t}\circ g_{x_d}^{-1}(x,y), \\
& \left. y_d-\int_0^t\pde {K_{x_d+s}}{x_d}\circ g_{x_d+t}\circ g_{x_d}^{-1}(x,y) \,ds\right).
\end{split}
\end{equation*}
Using estimates in the proof of Lemma~\ref{lemma est vf}, one gets that the increment in the last coordinate for $t\in[0,1]$ is bounded from above by
$$
\norm{ \pde{K}{x_d} }0
\leq \rho \|X_K\|_{C^0}
\leq \nu\rho
$$
as long as $\norm{g-\id}1$ is small.
Finally, the time-1 flow acts on the transversal $\{(x,0,y,0)\}$ by
\begin{equation*}
\varphi_{\widetilde H_0}^1(x,0,y,0)
= \left(\pi_1g(x,y),  1, \pi_2g(x,y), 
-\int_0^1\pde {K_{s}}{x_d}\circ g(x,y)\,ds \right).
\end{equation*}
In particular, if $g(0)=(0)$, $\varphi_{\widetilde H_0}^1(0)=(0,1,0,0)$ because $\pde{}{x_d}K(0,0)=0$.

%%%%%%%%%%%%%%%%%%%%%%%%%%%%%%%%%%%%%%%%%%
\section{Proof of Theorems~\ref{New} and~\ref{New2}}\label{end}

The proof of Theorem~\ref{New2} follows from the following steps:
\begin{enumerate}
\item 
Since elliptic closed orbits are stable, we can find a $C^\infty$ approximation $\widetilde H$ keeping the same (i.e. its analytic continuation) elliptic closed orbit.
\item 
Consider the $C^\infty$ Poincar\'e map $f$ of $\varphi_{\widetilde H}^t$ on a transversal to the elliptic closed orbit restricted to an energy surface.
\item
Use Theorem~\ref{GT} to obtain a $C^\infty$-symplectomorphim $\widetilde f$ close to $f$ with a homoclinic tangency.
\item 
Finally, Theorem~\ref{thm suspension} allows us to construct a Hamiltonian $C^2$-close to $\widetilde H$, which realizes the Poincar\'e map $\widetilde f$ on the energy surface.
\end{enumerate}

Assume that the energy level $H^{-1}(\{H(p)\})$ is far from Anosov.
The proof of Theorem~\ref{New} follows from Theorem~\ref{New2} after applying Theorem~\ref{BD} that gives elliptic closed orbits for some Hamiltonian $C^2$-close.

\medskip

Finally, we would like to mention a possible alternative strategy to prove Theorem~\ref{New2} in the absence of Theorem~\ref{GT}.
We first observe that an area-preserving diffeomorphism yielding an irrational invariant curve can be perturbed in order to create homoclinic tangencies, as proved in~\cite{MR}.
So, starting from a Hamiltonian with an elliptic closed orbit, one can perturb its tangent map and get a new Hamiltonian (using a version of Franks Lemma\cite{V}) whose Poincar\'e map is an area-preserving map satisfying a twist condition along a diophantine invariant curve.
KAM theory then assures us the stability of this structure, and a suspension of the result in~\cite{MR} holds homoclinic tangencies for a nearby Hamiltonian.

%%%%%%%%%%%%%%%%%%%%%%%%%%%%%%%%%%%%%%%%%
\section*{Acknowledgements}

MB was partially supported by Funda\c c\~ao para a Ci\^encia e a Tecnologia through SFRH/BPD/20890/2004 and the project PTDC/MAT/099493/2008. JLD was partially supported by Funda\c c\~ao para a Ci\^encia e a Tecnologia through the project ``Randomness in Deterministic Dynamical Systems and Applications'' PTDC/MAT/105448/2008.

%%%%%%%%%%%%%%%%%%%%%%%%%%%%%%%%%%%%%%%%%

\end{document}